\documentclass{article}
\usepackage[a4paper,margin=25mm]{geometry}

\usepackage{amsfonts}
\usepackage{amsmath}
\usepackage{amsrefs}
\usepackage{indentfirst}

\DeclareMathOperator{\Ass}{Ass}
\DeclareMathOperator{\Spec}{Spec}
\DeclareMathOperator{\depth}{depth}
\DeclareMathOperator{\ann}{ann}
\DeclareMathOperator{\Supp}{Supp}
\DeclareMathOperator{\ara}{ara}
\DeclareMathOperator{\divisor}{div}
\DeclareMathOperator{\hight}{ht}
\DeclareMathOperator{\Hom}{Hom}
\DeclareMathOperator{\ch}{char}

\usepackage{amsthm}
\theoremstyle{definition}
\newtheorem{definition}{Definition}[section]
\newtheorem{question}[definition]{Question}
\newtheorem{remark}[definition]{Remark}
\newtheorem{example}[definition]{Example}

\theoremstyle{plain}
\newtheorem{proposition}[definition]{Proposition}
\newtheorem{theorem}[definition]{Theorem}
\newtheorem{lemma}[definition]{Lemma}
\newtheorem{corollary}[definition]{Corollary}

\title{Finiteness of the set of associated primes for local cohomology modules of ideals via properties of almost factorial rings}

\author{Ryotaro Hanyu}
\date{}

\begin{document}
\maketitle

\begin{abstract}
We investigate the finiteness of the set of associated primes for local cohomology modules $H_I^{i}(J)$ of an ideal $J$ generated by an $R$-sequence, through the comparison of $H_I^{d+1}(J)$ and $H_I^d(R/J)$, where  $d = \depth_I(R)$. The properties of almost factorial rings play a key role in enabling this comparison. Under suitable conditions, we prove that the finiteness of $\Ass H_I^{d+1}(J)$ is equivalent to that of $\Ass H_I^d(R/J)$. Moreover, we give a few conditions under which the finiteness of $\Ass H_I^i(J)$ holds for all $i$.
\end{abstract}

\noindent\textbf{\small Keywords:} {\small Local cohomology, Associated primes}\\
\noindent{\small\textbf{MSC:} 13D45, 13A15, 13F05}

\section{Introduction}
Let $I$ be an ideal of a Noetherian ring $R$. Although a local cohomology module $H_I^i(R)$ is not necessarily finite over $R$, we may still ask whether the set $\Ass H_I^i(R)$ of associated primes of $H_I^i(R)$ is finite, as this can provide insights into the finiteness properties of local cohomology modules. Indeed, the set $\Ass M$ is finite for any finite $R$-module $M$. There are examples in which $\Ass H_I^i(R)$ is infinite : one described by Katzman\cite{Kat} in the context of $R$ being a local algebra over a field, and another given by Singh and Swanson\cite[Theorem 5.4]{SiSw} in the context of $R$ being a UFD algebra over a field. In contrast, there are examples in which $\Ass H_I^i(R)$ is finite. For instance, Huneke and Sharp\cite[Corollary 2.3]{HuSh} proved that $\Ass H_I^i(R)$ is finite for any ideal $I$ and any $i\geq0$ under the condition that $R$ is a regular ring of prime characteristic $p>0$, and Lyubeznik\cite[Corollary 3.6(c)]{Lyu1} proved the same result under the condition that $R$ is a regular semi-local ring containing a field. 

\begin{definition}
Let $R$ be a Noetherian ring. If $\Ass H_I^i(R)$ is finite for any ideal $I$ of $R$ and any $i\geq0$, $R$ is called \textit{LC-finite}.
\end{definition}

As mentioned above, regular rings of characteristic $p>0$, and regular semi-local rings containing a field are LC-finite. See \cite{BBLSZ} and \cite{Lyu3}, for more examples of an LC-finite ring.

Since the ring $R$ is generated by a single element as an $R$-module, one might expect that the local cohomology module $H_I^i(M)$ of a finite $R$-module $M$ has a more complicated structure than $H_I^i(R)$. The finiteness of $\Ass H_I^i(M)$ holds for $i=0$ and $i=1$. For $i=0$, this follows from the fact that $H_I^0(M)$ is finite over $R$. For $i=1$, the result was shown by Brodmann and Faghani\cite{BrFa}. For $i=2$, in the case where $M$ is an ideal, if $R$ is a locally almost factorial Noetherian normal ring, then $\Ass H_I^2(J)$ is finite for any ideal $I$, $J$ of $R$ \cite[Theorem 2.1]{Lew}. Under the same condition, there is an example in which $\Ass H_I^3(J)$ is an infinite set \cite[Section 3]{Lew}. In this example, the infiniteness of $\Ass H_I^3(J)$ is deduced from the infiniteness of $\Ass H_I^2(R/J)$ which is described in \cite{Kat}. To be precise, in this example, $H_I^2(R/J)$ is isomorphic to a submodule of $H_I^3(J)$ and all of the associated primes of $H_I^2(R/J)$ are inherited by $H_I^3(J)$. 

See \cite{DM} and \cite{Mar}, for more examples of an $R$-module $M$ such that $\Ass H_I^i(M)$ is finite.

We investigate the finiteness of the set of associated primes for local cohomology modules of ideals, based on the finiteness of that for local cohomology modules of Noetherian rings. The background involves the following question that was studied in \cite{Lew}:

\begin{question}
Let $R$ be an LC-finite regular ring, $J$ an ideal of $R$ generated by an $R$-sequence, and $I$ an ideal of $R$ containing $J$. Fix $i\geq1$. Does the finiteness of $\Ass H_I^{i-1}(R/J)$ imply the finiteness of $\Ass H_I^{i}(J)$?
\end{question}

For $i\le2$, it is true. If $R/J$ is a locally almost factorial normal ring, it is true for $i\le4$ \cite[Theorem 4.2]{Lew}. Question 1.2 is reduced to the following question, see \cite[Lemma 4.1 and Corollary 4.3]{Lew}: 

\begin{question}
Let $R$ be an LC-finite regular ring, $J$ an ideal of $R$ generated by an $R$-sequence of length at least $2$, and $I$ an ideal of $R$ containing $J$. Fix $i\geq1$, and suppose $\depth_I(R)=i-1$. Does the finiteness of $\Ass H_I^{i-1}(R/J)$ imply the finiteness of $\Ass H_I^{i}(J)$?
\end{question}

We consider cases according to the value of  $\depth_I(R/J)$, and in each case, we give an affirmative answer to the question under suitable assumptions : For the case $\depth_I(R/J)\leq1$, the question is true if $R/J$ is a locally almost factorial normal ring (see \cite[Theorem 4.2]{Lew} or Section 3). For the case $\depth_I(R/J)\geq2$, we give the following statement, which generalizes the case $\depth_I(R/J)=1$. This result will be proved in Section 4. 

\begin{theorem}
Let $R$ be a Noetherian ring, $J$ an ideal generated by an $R$-sequence of length at least $2$. Set $d= \depth_I(R)$, and let $I$ be an ideal containing $J$ such that $e=\depth_I(R/J)\geq2$. Suppose that $\Ass H_I^{d+1}(R)$ is finite and that there is an $R/J$-sequence $x_1,\ldots, x_{e-1}\in I$ such that $R/(x_1,\ldots, x_{e-1}, J)$ is a locally almost factorial normal ring. Then, 
\[\Ass H_I^{d+1}(J)\text{ is finite}\Leftrightarrow\Ass H_I^{d}(R/J)\text{ is finite.}\]
\end{theorem}

In Section 5, we consider the application of the theorem to ideals $I\supset J$, each of which is generated by a part of a regular system of parameters of a regular local ring $(R, \mathfrak{m})$. As a consequence, we obtain the following result, the proof of which is provided in Section 5:

\begin{theorem}
Let $(R, \mathfrak{m})$ be an LC-finite regular local ring and $I\supset J$ ideals of $R$ with $\depth_J(R)\geq2$. Suppose that each of $I$ and $J$ is generated by a part of a regular system of parameters. Then, for any $i\geq1$,
\[
\Ass H_I^i(J)\text{ is finite}\Leftrightarrow\Ass H_I^{i-1}(R/J)\text{ is finite.}
\]
\end{theorem}

Over Noetherian local rings $R$ of small dimension, Marley\cite{Mar} has shown several results regarding the finiteness of $\Ass H_I^i(M)$ for finite $R$-modules $M$. Using Marley's result, which establishes the finiteness of $\Ass H_I^{i-1}(R/J)$, we deduce the finiteness of $\Ass H_I^i(J)$ by applying the theorem above.
\section{Background}

\subsection{Associated primes}
Let $R$ be a ring, and $M$ an $R$-module. A prime ideal $P$ of $R$ is called an \textit{associated prime} of $M$ if $P=\ann(x)$ for some $x\in M$. The set of associated primes of $M$ is written $\Ass M$, or $\Ass_R M$.

\begin{remark}
Let $A\to B$ be a ring homomorphism, and $M$ a $B$-module. If $A\to B$ is surjective, $B=A/I$ for some ideal $I$ of $A$. $\Spec A/I$ is identified with the Zariski closed subset $V(I)$ of $\Spec A$, and $\Ass_{A/I} M = \Ass_A M$ as a subset of $\Spec A$. Let $S$ be a multiplicative set of $A$, and $B=A_S$. $\Spec B$ is identified with a subset of $\Spec A$. If A is Noetherian, $\Ass_{A_S} M = \Ass_A M$ as a subset of $\Spec A$ \cite[Theorem 6.2]{Mat}. 
\end{remark}

Let $R$ be a ring, $M$ an $R$-module, and $P\in\Ass M$. Since there exists an injection $A/P\hookrightarrow M$, we have an exact sequence $0\to A_{P}/PA_P\to M_P$, where $A_{P}/PA_P\neq0$. Thus,  $M_P\neq0$ and $P\in\Supp M$. This implies $\Ass M\subset\Supp M$.

\subsection{Local cohomology modules}

Let $R$ be a ring, $I$ an ideal of $R$, and $M$ an $R$-module. Set
\[H_I^0(M)=\lbrace x\in M\mid\text{there exists $n>0$ such that $I^{n}x=0$}\rbrace.\]
$H_I^0$ is a left exact functor. Let $H_I^i$ be the $i$-th right derived functor of $H_I^0$. $H_I^i(M)$ is called the $i$-th \textit{local cohomology module of M with respect to I}. 

From the definition, $H_I^i(M)=H_{\sqrt{I}}^i(M)$ and every element of $H_I^i(M)$ is annihilated by some power of $I$. If $P\in\Spec R$ does not contain $I$, a multiplicative set $R \setminus P$ contains some element $x\in I$. Since any element of $H_I^i(M)$ is annihilated by some power of $x$, $H_I^i(M)_P=0$. Therefore, $\Supp H_I^i(M)\subset V(I)$, and $\Ass H_I^i(M)\subset V(I)$.

Let $R$ be a Noetherian ring, $I=(x_1,\ldots, x_t)$ an ideal of $R$, and $M$ an \text{$R$-module}. A local cohomology module $H_I^i(M)$ is computed  as the $i$-th cohomology of the \v{C}ech complex, which has nonzero terms only in the range $0\le i\le t$, see \cite[Theorem A1.3]{Eis}. From this, the following holds:

\begin{proposition}
Let $R$ be a Noetherian ring, $I=(x_1,\ldots, x_t)$ an ideal of $R$, and $M$ an $R$-module. If $i>t$, then $H_I^i(M)=0$.
\end{proposition}
\begin{flushright}
$\square$
\end{flushright}

For an ideal $I$ of a Noetherian ring $R$, the \textit{arithmetic rank} of $I$, written $\ara(I)$, is defined by 
\[
\ara(I) = \text{min}\lbrace n\geq0\mid\text{there exists $x_1,\ldots, x_n$ such that $\sqrt{I}=\sqrt{(x_1,\ldots, x_n)}$}\rbrace.
\]
For the zero ideal, set $\ara(0)$ to be $0$. The proposition states that if $R$ is Noetherian, then for any ideal $I$ of $R$ and any $R$-module $M$, $H_I^i(M)=0$ for all $i>\ara(I)$.

\begin{lemma}
For ideals $I_1, \ldots, I_t$ of a Noetherian ring $R$, 
\[\sqrt{I_1\cap\cdots\cap I_t}=\sqrt{I_1\cdots I_t}.\]
\end{lemma}

\begin{proof}
By induction on $t$, it suffices to show that $\sqrt{I_1\cap I_2}=\sqrt{I_1I_2}$. It is obvious $\sqrt{I_1I_2}\subset\sqrt{I_1\cap I_2}$. If $a\in\sqrt{I_1\cap I_2}$, $a^n\in I_1\cap I_2$ for some $n>0$. Then, $a^{2n}=a^na^n\in I_1I_2$, and $a\in\sqrt{I_1I_2}$.
\end{proof}

The following two results are obtained since a local cohomology module can be computed by using the \v{C}ech complex, see \cite[Proposition 2.14]{Hun}.

\begin{proposition}
 Let $R$ be a Noetherian ring, $I$ an ideal of $R$, and $M$ an $R$-module. If $R\to S$ is a flat ring homomorphism, $H_I^i(M)\otimes_{R}S=H_{IS}^i(M\otimes_{R}S)$.
\end{proposition}
\begin{flushright}
$\square$
\end{flushright}

\begin{proposition}
Let $R$ be a Noetherian ring, $I$ an ideal of $R$, and $R\to S$ a ring homomorphism. Let $M$ be a $S$-module, then $H_I^i(M)=H_{IS}^i(M)$.
\end{proposition}
\begin{flushright}
$\square$
\end{flushright}

The following theorem exhibits the relationship among dimension, depth and the vanishing of a local cohomology module, see \cite[Theorem 5.8]{Hun}.

\begin{theorem}
Let $R$ be a Noetherian ring, $I$ an ideal of $R$, and $M$ a finite $R$-module. Then, $H_I^i(M)=0$, for $i<\depth_I(M)$, and for $i>\dim M$.
\end{theorem}
\begin{flushright}
$\square$
\end{flushright}

\subsection{Almost factorial rings}

\begin{definition}[{\cite[Section 12]{Mat}}]
Let $A$ be an integral domain and $K$ its field of fractions. Let $K^{\ast}$ be the multiplicative group of $K$. Let $\mathcal{F}=\lbrace R_\lambda \rbrace_{\lambda\in\Lambda}$ be a family of DVRs of $K$ and $v_\lambda$ the normalized valuation corresponding to $R_\lambda$. If the following two conditions hold, $A$ is called a \textit{Krull ring}, and $\mathcal{F}$ is called a defining family of $A$ :
\begin{enumerate}
\item $A=\underset{\lambda\in\Lambda}\bigcap R_{\lambda}$.
\item for every $x\in K^{\ast}$, there are at most a finite number of $\lambda\in\Lambda$ such that $v_\lambda(x)\neq0$.
\end{enumerate}
\end{definition}

A Noetherian normal domain is known to be a Krull ring. Let $\mathcal{F}$ be a defining family of a Krull ring $A$, and $P$ a height 1 prime ideal of $A$. Then, $A_P\in\mathcal{F}$ (see \cite[Theorem 12.3]{Mat}), and $A_P$ is a DVR. Let $\mathcal{P}$ be the set of height 1 prime ideals of $A$. Let $D(A)$ be the free Abelian group generated by $\mathcal{P}$, that is, every element of $D(A)$ is written as a formal sum 
\[\sum_{P\in\mathcal{P}}n_{P}\cdot P,\quad n_{P}\in\mathbb{Z},\]
where all but finitely many $n_P=0$ and addition is defined by
\[
\sum_{P\in\mathcal{P}}n_{P}\cdot P+\sum_{P\in\mathcal{P}}n^{\prime}_{P}\cdot P=\sum_{P\in\mathcal{P}}(n_{P}+n^{\prime}_{P})\cdot P.
\]
Let $K$ be the field of fractions of $A$ and $K^{\ast}$ the multiplicative group of $K$. For $a\in K^{\ast}$, set $\divisor(a) =\sum_{P\in\mathcal{P}}v_{P}(a)\cdot P$, where $v_P$ is the normalized valuation corresponding to $A_P$. div : $K^{\ast}\to D(A)$ is a group homomorphism. The \textit{divisor class group} of $A$, written $C(A)$, is defined by $C(A) = D(A)/\divisor(K^{\ast})$. In fact, $A$ is a UFD if and only if $C(A)$=0.

\begin{definition}
Krull ring $A$ is said to be \textit{almost factorial} if $C(A)$ is a torsion group, that is, every element of $C(A)$ has finite order.
\end{definition}

Note that UFDs are almost factorial. The following property of almost factorial rings enables us to rewrite local cohomology modules into more tractable form, which facilitates their computation (see Lemma 3.5, Proposition 3.6).

\begin{proposition}
Let $A$ be an almost factorial Krull ring. If $I$ is an ideal of pure hight $1$, that is, every minimal prime ideal of $I$ has height $1$, then $\sqrt{I}=\sqrt{aA}$ for some $a\in A$. 
\end{proposition}

\begin{proof}
Let $P_1, \ldots, P_t$ be all the minimal prime ideals of $I$, then $\sqrt{I}=P_1\cap\cdots\cap P_t$. For an each $P_i$, we show that $P_i=\sqrt{a_iA}$ for some $a_i\in A$. Since $I$ is of pure hight 1, $P_i$ is a hight 1 prime ideal of $A$. Let $K$ be the field of fractions of $A$ and $K^{\ast}$ the multiplicative group of $K$. Since $C(A)$ is torsion, $n\cdot P_i\in \divisor(K^{\ast})$ for some $n\in\mathbb{N}_{>0}$, that is, for some $a_i\in K^{\ast}$,
\[n\cdot P_i=\divisor(a_i)=\sum_{P\in\mathcal{P}}v_{P}(a_i)\cdot P.\]
It means that for $P\in\mathcal{P}$,
\[v_P(a_i)=\begin{cases}n &\text{if $P=P_i$}\\0 &\text{if $P\neq P_i$}\end{cases}.\]
Then, $a_i\in P_iA_{P_i}$ and $a_i\in A_P\setminus PA_P$ if $P\neq P_i$. In particular, $a_i\in\bigcap_{P\in\mathcal{P}}A_P$. Since the family $\lbrace A_P\mid P\in\mathcal{P}\rbrace$ of DVRs is a defining family of $A$, $a_i\in P_iA_{P_i}\cap A=P_i$. Thus, $\sqrt{a_iA}\subset P_i$. Furthermore,
\[a_iA_P=\begin{cases}P_i^nA_{P_i} &\text{if $P=P_i$}\\A_P &\text{if $P\neq P_i$}\end{cases}.\]
It follows that
\[a_iA=\bigcap_{P\in\mathcal{P}}(a_iA_P\cap A)=P_i^nA_{P_i}\cap A.\]
If $x\in P_i$, then $x^n\in a_iA$. Hence, $P_i=\sqrt{a_iA}$. By Lemma 2.3,
\[\sqrt{I}=\sqrt{a_1A}\cap\ldots\cap\sqrt{a_tA}=\sqrt{a_1\cdots a_tA}.\]
\end{proof}

\section{Theoretical tools}

In this section, we briefly summarize the Lewis's approach to Question 1.2, which forms the basis for our discussion. We restate the question as follows: Let $R$ be an LC-finite regular ring, $J$ an ideal of $R$ generated by an $R$-sequence, and $I\supset J$. Fix $i\geq1$. Does the finiteness of $\Ass H_I^{i-1}(R/J)$ imply the finiteness of $\Ass H_I^{i}(J)$?

Since both $\Ass H_I^0(R/J)$ and $\Ass H_I^1(J)$ are finite, the question is true when $i=1$. Let $j>0$ be the length of an $R$-sequence generating $J$. In the case $j=1$, as an $R$-module, $J\simeq R$ and $H_I^i(J)\simeq H_I^i(R)$. The LC-finiteness of $R$ implies that $\Ass H_I^i(J)$ is finite, and the claim is trivial. We may assume $j\geq2$. 

For a Noetherian normal ring $R$, a local ring $R_P$ at $P\in\Spec R$ is a Noetherian normal domain, hence a Krull ring. If $R_P$ is almost factorial for any $P\in\Spec R$, $R$ is said to be \textit{locally almost factorial}.

Regarding the finiteness of $\Ass H_I^i(J)$ for locally almost factorial rings, we have the following result from \cite{Lew}.

\begin{theorem}[{\cite[Theorem 2.1]{Lew}}] 
Let $R$ be a locally almost factorial Noetherian normal ring, and $I$, $J$ be ideals of $R$. The set $\Ass H_I^2(J)$ is finite.
\end{theorem}
\begin{flushright}
$\square$
\end{flushright}

Since a regular ring is locally almost factorial, Theorem 3.1 implies that $\Ass H_I^2(J)$ is always finite in the above question. Since $\Ass H_I^1(R/J)$ is finite, the question is true when $i=2$. The assertions for $i=3, 4$ also hold under the hypothesis that $R/J$ is normal and locally almost factorial \cite[Theorem 4.2(iii)]{Lew}. The proof begins by reducing the problem to the case $\depth_I(R)\geq i-1$, using a generalized version of Hellus's isomorphism (see \cite[Theorem 4.1 and Corollary 4.3]{Lew}). Furthermore, we use the following lemma from \cite{Lew}.

\begin{lemma}[{\cite[Lemma 4.1]{Lew}}]
Let $R$ be a Noetherian ring and $I$, $J$ be ideals of $R$. Fix $i\geq1$ and assume $\Ass H_I^i(R)$ is finite. If $\depth_I(R)>i-1$, then $\Ass H_I^i(J)$ is finite if and only if $\Ass H_I^{i-1}(R/J)$ is finite.
\end{lemma}
\begin{flushright}
$\square$
\end{flushright}

By the lemma, we should view only the case $\depth_I(R)=i-1$. Since 
\[i-1=\depth_I(R)=\depth_I(R/J)+\depth_J(R),\]
and $j=\depth_J(R)\geq2$, there are the only cases 
\[
(\depth_I(R/J), \depth_J(R))=
\begin{cases} (0, 2)&\text{for $i=3$}\\(1, 2), (0, 3) &\text{for $i=4$}\end{cases}.
\]
For the case $\depth_I(R/J)=0$ and the case $\depth_I(R/J)=1$, the following results are found in the proof of \cite[Theorem 4.2]{Lew}:

\begin{theorem}[{cf. \cite[Theorem 4.2]{Lew}}]
Let $R$ be a Noetherian ring, $J$ an ideal generated by an $R$-sequence of length at least $2$, and $I$ an ideal containing $J$. Set $d=\depth_I(R)$. Suppose that $\Ass H_I^{d+1}(R)$ is finite. Then, in the following two cases, $\Ass H_I^{d+1}(J)$ is finite if and only if $\Ass H_I^{d}(R/J)$ is finite. 
\begin{enumerate}
\item In the case $\depth_I(R/J)=0$, the claim is true if $R/J$ is Cohen-Macaulay and locally a domain.
\item In the case $\depth_I(R/J)=1$, the claim is true if $R/J$ is a locally almost factorial normal ring. 
\end{enumerate}
\end{theorem}

Later, we generalize Case 2 by following Lewis's approach in \cite{Lew} and incorporating the method to remove the restriction on $\depth_I(R/J)$. For the sake of discussion, the proof of Theorem 3.3 is given below. First, we prepare preliminary results.

\begin{lemma}
Let $A$ be a Cohen-Macaulay ring with $\dim A>0$, and $I$ an ideal of $A$ with $\depth_I(A)=0$. If $A$ is locally a domain, $\sqrt{I}=\sqrt{aA}\cap I_0$ for some $a\in A$ and for some ideal $I_0$ with $\depth_{I_0}(A)>0$.
\end{lemma}

\begin{proof}
Since $\hight I=\depth_I(A)=0$, we can write $\sqrt{I} =J\cap I_0$, in which $J$ is a finite intersection of minimal prime ideals of $A$, and $I_0$ is a finite intersection of height $>0$ prime ideals. Thus, $\depth_{I_0}(A)=\hight I_0>0$. To see that $J$ is a radical of some principal ideal of $A$, we show that every minimal prime ideal of $A$ is a radical of some principal ideal of $A$. Let $P_1,\ldots, P_t$ be all the minimal prime ideals of $A$. Since $A$ is Noetherian and locally a domain, 
\[A\simeq A/P_1\times\cdots\times A/P_t.\]
Let $e_i\in A$ be the element corresponding to ``the $i$-th standard idempotent'' $(0,\ldots, 0, 1, 0,\ldots,0)\in A/P_1\times\cdots\times A/P_t$. Each $P_i$ corresponds to the ideal 
\[A/P_1\times\cdots\times A/P_{i-1}\times (0)\times A/P_{i+1}\times\cdots\times A/P_t.\]
Thus, $e_i\notin P_i$, and $e_j\in P_i$ if $j\neq i$. In addition, $e_1+\cdots+e_t=1$, since this element corresponds to $(1,\ldots, 1)$. Hence, $1-e_i\in P_i$, and it follows that $P_i$ is the only minimal element of $V((1-e_i)A)$. Indeed, if not, there exists $P\in V((1-e_i)A)$ such that $P\supset P_j$ for some $j\neq i$. Then, $1-e_i, e_i\in P$, and $1\in P$, yielding a contradiction. Since $J$ is a finite intersection of minimal prime ideals, we see that $J$ is of the form,
\[J=\bigcap_i\sqrt{a_iA}=\sqrt{aA},\quad\text{where $a=\prod_ia_i$}.\]
\end{proof}

In the case where the depth equals $1$, we have the following result from \cite{Lew}.

\begin{lemma}[{\cite[Corollary 2.1]{Lew}}]
Let $A$ be a locally almost factorial Noetherian normal ring, and $I$ an ideal of $A$ such that $\depth_I(A)=1$. Then, there is an ideal $I_0$ of $A$ such that 
\[\depth_{I_0}(A)>1, \]
and a finite open covering $\lbrace\Spec(A_{f_1}),\ldots,\Spec(A_{f_t})\mid f_1,\ldots, f_t\in A\rbrace$ of $\Spec A$ such that for each $i=1,\ldots, t$,
\[\sqrt{IA_{f_i}}=\sqrt{a_iA_{f_i}}\cap I_0A_{f_i}\quad\text{for some $a_i\in A$}.\]
\end{lemma}
\begin{flushright}
$\square$
\end{flushright}

The following result from \cite{Lew} is useful when computing the local cohomology modules of ideals of a locally almost factorial ring.
\begin{proposition}[{\cite[Corollary 2.2]{Lew}}]
Let $A$ be a Noetherian ring, $a\in A$, and $I_0$ an ideal of $A$. Let $I=\sqrt{aA}\cap I_0$. There is a natural isomorphism of functors $H_I^i(-)\simeq H_{I_0}^i(-)_a$ for all $i\geq2$.
\end{proposition}
\begin{flushright}
$\square$
\end{flushright}

\begin{lemma}
Let $\psi : M\to N$ be a homomorphism of $A$-modules, and $B$ an $A$-algebra. If $N$ also has a $B$-module structure, then $\psi$ factors through the natural map $M\to B\otimes_AM$, $x\mapsto 1\otimes x$:
\[M\xrightarrow{\psi}N\quad=\quad M\to B\otimes_AM\to N.\]
\end{lemma}

\begin{proof}
By Hom-Tensor adjointness,
\[
\Hom_B(B\otimes_AM, N)\simeq\Hom_A(M, \Hom_B(B, N))\simeq\Hom_A(M, N). 
\]
Under the first isomorphism, $\tilde{f} : B\otimes_AM\to N$ corresponds to the map
\[F : M\to\Hom_B(B, N),\quad x\mapsto (F_x \colon B\ni b\mapsto \tilde{f}(b\otimes x)).
\]
The second isomorphism is induced by 
\[\Hom_B(B, N)\simeq N,\quad g\mapsto g(1).
\]
Under the composition of these isomorphisms, $\tilde{f} \colon B\otimes_AM\to N$ corresponds to the map $f \colon M\to N, x\mapsto \tilde{f}(1\otimes x)$. Hence, there is a map $\tilde{\psi} \colon B\otimes_AM\to N$ such that 
\[M\xrightarrow{\psi}N\quad=\quad M\to B\otimes_AM\xrightarrow{\tilde{\psi}} N.
\]
\end{proof}

\begin{proof}[Proof of Theorem 3.3]
Set $S=R/J$. In Case 1, by Lemma 3.4, 
\[\sqrt{I/J}=\sqrt{a(R/J)}\cap (I_0/J)\] 
for some $a\in R$ and for some ideal $I_0\supset J$ of $R$ with $\depth_{I_0}(R/J)>0$. Then, $\sqrt{IS}=\sqrt{aS\cap I_0S}$ and $H_I^i(S)=H_{aS\cap I_0S}^i(S)$. By Proposition 3.6, \text{$H_I^i(S)\simeq H_{I_0}^i(S)_a$} for all $i\geq 2$. Now, 
\[d=\depth_I(R)=\depth_I(R/J)+\depth_J(R)=0+\depth_J(R)\geq2,\]
and $H_I^{d}(S)$ has a $S_a$-module structure. The natural map $\psi \colon H_I^{d}(R)\to H_I^{d}(S)$ factors through $S_a\otimes_RH_I^{d}(R)$ by Lemma 3.7. We want to show $\psi=0$. It suffices to show that $S_a\otimes_RH_I^{d}(R)=0$. Since
\[S_a\otimes_RH_I^{d}(R)\simeq S_a\otimes_{R_a}R_a\otimes_RH_I^{d}(R)\simeq S_a\otimes_{R_a}H_{IR_a}^{d}(R_a),\]
it suffices to show $H_{IR_a}^{d}(R_a)=0$. The following two equalities hold: 
\[
\sqrt{I}=\sqrt{(aR+J)\cap I_0},\quad\sqrt{IR_a}=\sqrt{(aR+J)R_a\cap I_0R_a}=\sqrt{I_0R_a}.
\]



By the second equality, $H_{IR_a}^{d}(R_a)\simeq H_{I_0R_a}^{d}(R_a)$. Now $J$ is generated by an $R$-sequence and contained in $I_0$,
\[
\depth_{I_0R_a}(R_a)\geq\depth_{I_0}(R)=\depth_{I_0}(R/J)+\depth_J(R)> d.
\]
Therefore, $H_{IR_a}^{d}(R_a)=0$ and $\psi =0$. From the long exact sequence of a local cohomology,
\[0\to H_I^d(S)\to H_I^{d+1}(J)\to H_I^{d+1}(R)\]
is exact, and
\[\Ass H_I^d(S)\subset\Ass H_I^{d+1}(J)\subset\Ass H_I^d(S)\cup\Ass H_I^{d+1}(R).\]
By the assumption $\Ass H_I^{d+1}(R)$ is finite, the claim follows.

In Case 2, by Lemma 3.5, there is an ideal $I_0$ of $R$,
\[I_0\supset J\quad\text{with $\depth_{I_0}(R/J)>1$},\] 
and a finite open covering of $\Spec S$, 
\[\lbrace\Spec(S)\cap D(f_1),\ldots,\Spec(S)\cap D(f_t)\mid f_1,\ldots, f_t\in R\rbrace\] 
such that for each $i$,
\[
\sqrt{IR_{f_i}/JR_{f_i}}=\sqrt{a_i(R_{f_i}/JR_{f_i})}\cap (I_0R_{f_i}/JR_{f_i})\quad\text{for some $a_i\in R$},
\]
where $D(f_i) = \Spec R_{f_i}$ is the principal open subset of $\Spec R$. For any $R$-module $M$, $\Ass H_I^j(M)\subset V(I)\subset V(J) = \Spec S$. In addition, for each $i$,
\[\Ass H_I^j(M_{f_i})=\Ass H_I^j(M)\cap D(f_i).\]
It follows that $\Ass H_I^j(M)$ is finite if and only if $\Ass H_I^j(M_{f_i})$ is finite for all $i=1,\ldots, t$. Then, we should prove that for all $i=1,\ldots, t$,
\[
\Ass H_I^{d}(S_{f_i})\text{ is finite}\Leftrightarrow\Ass H_I^{d+1}(JR_{f_i})\text{ is finite}.
\]
Now, $d=1+\depth_J(R)\geq3$. Replacing $R$ by $R_{f_i}$ and $a$ by $a_i$ in the discussion of Case 1, we see that it suffices to show that $H_{I(R_{f_i})_{a_i}}^d((R_{f_i})_{a_i})=0$, and we can also conclude $\sqrt{I(R_{f_i})_{a_i}}=\sqrt{I_0(R_{f_i})_{a_i}}$. Then,
\[H_{I(R_{f_i})_{a_i}}^d((R_{f_i})_{a_i})\simeq H_{I_0(R_{f_i})_{a_i}}^d((R_{f_i})_{a_i})=H_{I_0}^d((R_{f_i})_{a_i}).
\]
It follows that $H_{I(R_{f_i})_{a_i}}^d((R_{f_i})_{a_i})=0$, since
\[\depth_{I_0}((R_{f_i})_{a_i})\geq\depth_{I_0}(R)=\depth_{I_0}(R/J)+\depth_J(R)>d.
\]   
\end{proof}

\begin{theorem}[{\cite[Theorem 4.2]{Lew}}]
Let $R$ be an LC-finite regular ring, $J$ an ideal generated by an $R$-sequence of length at least $2$. Let $I$ be an ideal of $R$ containing $J$. Suppose that $R/J$ is normal and locally almost factorial. Then, for any $i\leq4$, $\Ass H_I^i(J)$ is finite if and only if $\Ass H_I^{i-1}(R/J)$ is finite.
\end{theorem}

\begin{proof}
As we have already seen at the beginning of this section, we only need to consider the following cases, where $\depth_I(R/J)=0, 1$: 
\[
(\depth_I(R/J), \depth_J(R))=(0, 2), (0, 3), (1, 2).
\]
$R/J$ is Cohen-Macaulay, since $R$ is regular and $J$ is generated by an $R$-sequence. In addition, by the assumption that $R/J$ is normal, $R/J$ is locally a domain. For all of the cases above, the claim holds by Theorem 3.3.
\end{proof}

\section{Main results}
In this section, we consider the case $\depth_I(R/J)\geq2$. To handle this case in accordance with Lewis's approach, we assume the existence of a suitable $R/J$-sequence in $I$. The assumption in the following statement coincides the assumption in Case 2 of Theorem 3.3, when $\depth_I(R/J)=1$:

\begin{proposition}
Let $R$ be a Noetherian ring, $J$ an ideal generated by an $R$-sequence of length at least $2$. Set $S=R/J$. Let $I$ be an ideal containing $J$ such that $e=\depth_I(R/J)\geq2$. Suppose that there is an $R/J$-sequence $x_1,\ldots, x_{e-1}\in I$ such that $R/(x_1,\ldots, x_{e-1}, J)$ is a locally almost factorial normal ring. Then, there is an ideal $I_0\supset (x_1,\ldots, x_{e-1}, J)$ of $R$ such that 
\[\depth_{I_0}(S)>e\]
and a finite open covering $\lbrace V(I)\cap D(f_1),\ldots,\ V(I)\cap D(f_t)\mid f_1,\ldots, f_t\in R\rbrace$ of $V(I)$ such that for each $i=1,\ldots, t$
\[\sqrt{IS_{f_i}}=\sqrt{(x_1,\ldots, x_{e-1}, y_i)S_{f_i}}\cap I_0S_{f_i}=\sqrt{y_iS_{f_i}\cap I_0S_{f_i}+(x_1,\ldots, x_{e-1})S_{f_i}}\]
for some $y_i\in R$. Where, $D(f_i)$ is the principal open subset of $\Spec R$. 
\end{proposition}

\begin{proof}
With the notation $\underline{x}=(x_1,\ldots, x_{e-1})$, the following equality holds:
\[\depth_I(R/(\underline{x}, J))=\depth_I(R/J)-(e-1)=1\]
By Lemma 3.5, there is an ideal $I_0\supset (\underline{x}, J)$ of $R$ such that $\depth_{I_0}(R/(\underline{x}, J))>1$, and a finite open covering of $\Spec R/(\underline{x}, J)$
\[
\lbrace\Spec(R/(\underline{x}, J))\cap D(f_1),\ldots, \Spec(R/(\underline{x}, J))\cap D(f_t)\mid f_1,\ldots, f_t\in R\rbrace
\]
such that for each $i=1,\ldots, t$
\[
\sqrt{I(R/(\underline{x}, J))_{f_i}}=\sqrt{y_i(R/(\underline{x}, J))_{f_i}}\cap I_0(R/(\underline{x}, J))_{f_i}\quad\text{for some $y_i\in R$}.
\]
Then, $\depth_{I_0}(S)=(e-1)+\depth_{I_0}(R/(\underline{x}, J))>e$. The above formula is rewritten in the following form: $\sqrt{IS_{f_i}/\underline{x}S_{f_i}}=\sqrt{y_i(S_{f_i}/\underline{x}S_{f_i})}\cap (I_0S_{f_i}/\underline{x}S_{f_i})$. From this, we derive the first equality in the claim. Set $S_i=S_{f_i}$. Let $a\in\sqrt{IS_i}$, then,  $a \mod \underline{x}S_i\in\sqrt{IS_{f_i}/\underline{x}S_{f_i}}=\sqrt{y_i(S_i/\underline{x}S_i)}\cap (I_0S_i/\underline{x}S_i)$. It follows that $a\in I_0S_i$, and $a^n-y_ib\in\underline{x}S_i$ for some $n>0$ and $b\in S_i$. Thus, $a\in\sqrt{(\underline{x}, y_i)S_i}\cap I_0S_i$. Conversely, let $a\in\sqrt{(\underline{x}, y_i)S_i}\cap I_0S_i$. Then, $a \mod \underline{x}S_i\in\sqrt{y_i(S_i/\underline{x}S_i)}\cap (I_0S_i/\underline{x}S_i)=\sqrt{IS_i/\underline{x}S_i}$, and $a^n\in IS_i+\underline{x}S_i$ for some $n>0$. Since $(x_1,\ldots, x_{e-1})\subset I$, $a\in\sqrt{IS_i}$. Therefore, $\sqrt{IS_i}=\sqrt{(\underline{x}, y_i)S_i}\cap I_0S_i$.

Next, we prove the second equality in the claim, which, with the notation $S_i=S_{f_i}$, takes the following form:
\[\sqrt{(\underline{x},y_i)S_i}\cap I_0S_i=\sqrt{y_iS_i\cap I_0S_i+\underline{x}S_i}.
\]
Taking the radical of both ideals, it suffices to show that 
\[\sqrt{(\underline{x},y_i)S_i\cap I_0S_i}=\sqrt{y_iS_i\cap I_0S_i+\underline{x}S_i}.\]
For any $a$ in the left-hand side, $a^n=x_1b_1+\cdots+x_{e-1}b_{e-1}+y_ic\in I_0S_i$ for some $n>0$ and $b_1,\ldots, b_{e-1}, c\in S_i$. Here, $y_ic\in y_iS_i\cap I_0S_i$, since $x_1,\ldots x_{e-1}\in I_0$. Then, \text{$a^n\in y_iS_i\cap I_0S_i+\underline{x}S_i$}. Conversely, for any $a$ in the right-hand side, $a^n=x_1b_1+\cdots+x_{e-1}b_{e-1}+y_ic$ for some $n>0$, where $b_1,\ldots, b_{e-1}, c\in S_i$, and $y_ic\in I_0S_i$. Then, $a^n\in (\underline{x}, y_i)S_i\cap I_0S_i$, since $x_1, \ldots, x_{e-1}\in I_0$.

From the above, we have shown both the inequality $\depth_{I_0}(S)>e$ and the equality within the claim. Furthermore, since $\Spec R/(\underline{x}, J)\supset V(I)$,
\[\lbrace V(I)\cap D(f_1),\ldots,\ V(I)\cap D(f_t)\mid f_1,\ldots, f_t\in R\rbrace\]
is an open covering of $V(I)$, as desired.
\end{proof}

The relationship of the ideal's radical given in the proposition shows, through the form of the equation $\sqrt{IS_i}=\sqrt{y_iS_i\cap I_0S_i+(x_1,\ldots, x_{e-1})S_i}$, how far the structure of $IS_i$ is from that of $y_iS_i\cap I_0S_i$. In Lewis's approach, the ideal $y_iS_i\cap I_0S_i$ plays an important role in establishing the equivalence between the finiteness of $H_I^{d+1}(J)$ and that of $H_I^d(R/J)$. The difference between $IS_i$ and $y_iS_i\cap I_0S_i$, namely the $(x_1,\ldots, x_{e-1})S_i$ part, can be ignored for the purpose of computing the local cohomology modules relevant to the present discussion:

\begin{lemma}
Let $A$ be a Noetherian ring, and $I_0$ an ideal of $A$. An ideal $I$ of $R$ is assumed to be of the form  
\[\sqrt{I}=\sqrt{x_eA\cap I_0+(x_1,\ldots, x_{e-1})A}
\]
for some $x_1,\ldots, x_e\in A$ such that $(x_ex_1, x_ex_2, \ldots, x_ex_{e-1})A\subset I_0$. Then, $H_I^i(M)\simeq H_{x_eA\cap I_0}^i(M)$ for any $A$-module $M$ and any $i>e$.
\end{lemma}

\begin{proof}
By the assumption and with the notation $\underline{x}=(x_1,\ldots, x_{e-1})$, we have $H_I^i(M)\simeq H_{x_eA\cap I_0+\underline{x}A}^i(M)$. It follows from the Mayor-Vietoris sequence that
\[
H_{x_eA\cap I_0\cap\underline{x}A}^{i-1}(M)\to H_I^i(M)\to H_{x_eA\cap I_0}^i(M)\oplus H_{\underline{x}A}^i(M)\to H_{x_eA\cap I_0\cap\underline{x}A}^i(M)
\]
is exact. By the assumption that $(x_ex_1,\ldots, x_ex_{e-1})A\subset I_0$, 
\begin{equation*}\begin{split}
\sqrt{x_eA\cap I_0\cap\underline{x}A}&=\sqrt{I_0\cap\sqrt{x_eA\cap\underline{x}A}}=\sqrt{I_0\cap\sqrt{(x_ex_1,\ldots, x_ex_{e-1})A}}\\&=\sqrt{I_0\cap (x_ex_1,\ldots, x_ex_{e-1})A}=\sqrt{(x_ex_1,\ldots, x_ex_{e-1})A}
\end{split}\end{equation*}
The above exact sequence is rewritten in the following form:
\[
H_{(x_e\underline{x})}^{i-1}(M)\to H_I^i(M)\to H_{x_eA\cap I_0}^i(M)\oplus H_{\underline{x}A}^i(M)\to H_{(x_e\underline{x})}^i(M),
\]
where $(x_e\underline{x})$ denotes the ideal $(x_ex_1,\ldots, x_ex_{e-1})A$. If $i>e$, by Proposition 2.2,  this exact sequence is 
\[0\to H_I^i(M)\to  H_{x_eA\cap I_0}^i(M)\oplus 0\to 0.\]
 Hence, $H_I^i(M)\simeq H_{x_eA\cap I_0}^i(M)$ if $i>e$.
\end{proof}

At this point, we generalize Case 2 of Theorem 3.3.

\begin{theorem}
Let $R$ be a Noetherian ring, $J$ an ideal generated by an $R$-sequence of length at least $2$. Set $d= \depth_I(R)$, and let $I$ be an ideal containing $J$ such that $e=\depth_I(R/J)\geq2$. Suppose that $\Ass H_I^{d+1}(R)$ is finite, and that there is an $R/J$-sequence $x_1,\ldots, x_{e-1}\in I$ such that $R/(x_1,\ldots, x_{e-1}, J)$ is a locally almost factorial normal ring. Then, 
\[\Ass H_I^{d+1}(J)\text{ is finite}\Leftrightarrow\Ass H_I^{d}(R/J)\text{ is finite.}\]
\end{theorem}

\begin{proof}
By Proposition 4.1 and, with the notation $\underline{x}=(x_1,\ldots, x_{e-1})$ and $S=R/J$, there is an ideal $I_0\supset (x_1,\ldots, x_{e-1}, J)$ of $R$ such that 
\[\depth_{I_0}(S)>e,\]
and a finite open covering $\lbrace V(I)\cap D(f_1),\ldots,V(I)\cap D(f_t)\mid f_1,\ldots, f_t\in R\rbrace$ of $V(I)$ such that for each $i=1,\ldots, t$
\[
\sqrt{IS_{f_i}}=\sqrt{(\underline{x}, y_i)S_{f_i}}\cap I_0S_{f_i}=\sqrt{y_iS_{f_i}\cap I_0S_{f_i}+\underline{x}S_{f_i}}
\]
for some $y_i\in R$. By Lemma 4.2,
\[H_I^j({S_{f_i}})\simeq H_{y_iR\cap I_0}^j(S_{f_i})\quad\text{for all $j>e$}.\]
Since both $\Ass H_I^d(S)$ and $\Ass H_I^{d+1}(J)$ are contained in $V(I)$, we should prove that for $i=1, \ldots, t$,
\[
\Ass H_I^{d}(S_{f_i})\text{ is finite}\Leftrightarrow\Ass H_I^{d+1}(JR_{f_i})\text{ is finite}.
\]
By Proposition 3.6, $H_I^j({S_{f_i}})\simeq H_{I_0}^j(S_{f_i})_{y_i}$ for all $j>e$ $(\geq2)$. $H_I^d(S_{f_i})$ has a $(S_{f_i})_{y_i}$-module structure, since $d=e+\depth_{J}(R)\geq e+2$. As we have seen in the proof of Theorem 3.3, it suffices to show that $H_I^d((R_{f_i})_{y_i})=0$ and we can also conclude 
\[
\sqrt{I(R_{f_i})_{y_i}}=\sqrt{(\underline{x}, y_i, J)(R_{f_i})_{y_i}\cap I_0(R_{f_i})_{y_i}}=\sqrt{I_0(R_{f_i})_{y_i}}.
\]
It follows from this formula that $H_I^d((R_{f_i})_{y_i})\simeq H_{I_0}^d((R_{f_i})_{y_i})=0$, since
\[
\depth_{I_0}((R_{f_i})_{y_i})\geq\depth_{I_0}(R)=\depth_{I_0}(S)+\depth_J(R)>d.
\] 
\end{proof}

Constructing an explicit example in which the quotient ring is locally almost factorial appears to be quite subtle. Indeed, even when $J$ is generated by an $R$-sequence, it seems difficult in general to determine whether the quotient $R/J$ is normal or factorial. At present, the main situation where the technical assumptions of Theorem 4.3 can be verified is the case of a regular local ring, as treated in Section 5.

\begin{corollary}
Let $R$ be a Noetherian ring, $J$ an ideal generated by an $R$-sequence of length at least $2$. Set $d=\depth_I(R)$, and let $I$ be an ideal containing $J$ such that $e=\depth_I(R/J)\geq2$. Suppose that $\Ass H_I^{d+1}(R)$ are finite, and that there is an $R/J$-sequence $x_1,\ldots, x_{e-1}\in I$ such that $R/(x_1,\ldots, x_{e-1}, J)$ is a locally almost factorial normal ring. Then, for all $i\leq d+1$,
\[\Ass H_I^{i}(J)\text{ is finite}\Leftrightarrow\Ass H_I^{i-1}(R/J)\text{ is finite.}\]
\end{corollary}

\begin{proof}
For the case $i=d+1$, it follows by Theorem 4.3. If $i\leq d$,
\[0\to H_I^{i-1}(R/J)\to H_I^i(J)\to H_I^i(R)\]
is exact. If $i<d$, then $H_I^i(R)=0$, and the claim immediately follows. If $i=d$, the claim follows from the fact that $\Ass H_I^d(R)$ is finite. Indeed, since $d=\depth_I(R)$, $H_I^d(R)$ is the first nonzero local cohomology modules of $R$ with respect to $I$. Then $\Ass H_I^d(R)$ is finite by \cite[proposition 2.1]{BrFa}. 
\end{proof}

\section{Applications}

Let $(R, \mathfrak{m})$ be an $r$-dimensional regular local ring. $x_1,\ldots, x_r\in\mathfrak{m}$ is a \textit{regular system of parameters} if $\mathfrak{m}=(x_1,\ldots, x_r)$, or equivalently if the images of $x_1,\ldots, x_r$ in $\mathfrak{m}/\mathfrak{m}^2$ are linearly independent over $R/\mathfrak{m}$. For any ideal $I$ of $R$, $R/I$ is regular if and only if $I$ is generated by a part of a regular system of parameters. Let $I\supset J$ be ideals of $R$ with $\depth_J(R)\geq2$. When each of $I$ and $J$ is generated by a part of a regular system of parameters, we can take an $R/J$-sequence as stated in Theorem 4.3:

\begin{lemma}
Let $(R, \mathfrak{m}, k)$ be a regular local ring and $I\supset J$ ideals of $R$. If each of $I$ and $J$ is generated by a part of a regular system of parameters, there exists a subset of a regular system of parameters $x_1,\ldots,x_j,\ldots, x_d\in\mathfrak{m}$ such that $J=(x_1\ldots, x_j)$ and $I=(x_1,\ldots, x_j,\ldots, x_d)$. 
\end{lemma}

\begin{proof}
Since $(R/J)/(I/J)\simeq R/I$ is regular, $I/J$ is generated by a part of a regular system of parameters of $R/J$. We can take $y_1,\ldots, y_n\in I$ such that $I/J=(y_1,\ldots, y_n)R/J$, and the images of $y_1,\ldots, y_n$ in $(\mathfrak{m}/J)/(\mathfrak{m}/J)^2$ are linearly independent over $k$. Since $\mathfrak{m}/(\mathfrak{m}^2+J)\simeq(\mathfrak{m}/J)/(\mathfrak{m}/J)^2$ as a $k$-module, the images of $y_1,\ldots, y_n$ in $\mathfrak{m}/(\mathfrak{m}^2+J)$ are also linearly independent over $k$. Let $x_1,\ldots, x_j\in\mathfrak{m}$ be a part of a regular system of parameters which generates $J$. Then, $I=(x_1,\ldots, x_j, y_1,\ldots, y_n)$. It suffices to show that the images of $x_1,\ldots, x_j, y_1,\ldots, y_n$ in $\mathfrak{m}/\mathfrak{m}^2$ are linearly independent over $k$. Assume that $a_1,\ldots, a_j, b_1,\ldots, b_n\in R$, and
\[a_1x_1+\cdots+a_jx_j+b_1y_1+\cdots+b_ny_n\in\mathfrak{m}^2.\]
It follows from the choice of $x_1,\ldots,x_j$ that $a_1x_1+\cdots+a_jx_j\in J$, and $b_1y_1+\cdots+b_ny_n\in\mathfrak{m}^2+J$. Thus, $b_1,\ldots, b_n\in\mathfrak{m}$, since the images of $y_1,\ldots, y_n$ in $\mathfrak{m}/(\mathfrak{m}^2+J)$ are linearly independent over $R/\mathfrak{m}$. Then, $b_1y_1+\cdots b_ny_n\in\mathfrak{m}^2$, and $a_1x_1+\cdots+a_jx_j\in\mathfrak{m}^2$. Hence, $a_1,\ldots, a_j\in\mathfrak{m}$, since the images of $x_1,\ldots, x_j$ in $\mathfrak{m}/\mathfrak{m}^2$ are linearly independent over $R/\mathfrak{m}$.
\end{proof}

\begin{theorem}
Let $(R, \mathfrak{m})$ be an LC-finite regular local ring and $I\supset J$ ideals of $R$ with $\depth_J(R)\geq2$. Suppose that each of $I$ and $J$ is generated by a part of a regular system of parameters. Then, for any $i\geq1$,
\[
\Ass H_I^i(J)\text{ is finite}\Leftrightarrow\Ass H_I^{i-1}(R/J)\text{ is finite.}
\]
\end{theorem}

\begin{proof}
It is true for $i\leq2$, see Section 1. Fix $i\geq3$. As we have seen in Section 3, we may assume $\depth_I(R)=i-1\geq2$. Since $R/J$ is regular, for the cases $\depth_I(R/J)=0, 1$, the claim holds by Theorem 3.3. Fix $e=\depth_I(R/J)\geq2$. Then, there are elements $x_1,\ldots, x_{e-1}\in I$ such that $R/(x_1,\ldots, x_{e-1}, J)$ is regular. The claim also holds for the case $e\geq2$, by Theorem 4.3.
\end{proof} 

\begin{example}
Let $R=\mathbb{Z}_{p\mathbb{Z}}[[x_1, \ldots, x_n]]$, a formal power series ring over a discrete valuation ring $\mathbb{Z}_{p\mathbb{Z}}$. $R$ is a regular local ring with its maximal ideal $\mathfrak{m}=(p, x_1, \ldots, x_n)$ and its residue field $\mathbb{Z}/p\mathbb{Z}$. A ring like $R$ is called an unramified ring of unequal characteristic (see below for the precise definition). 

$R$ is known to be LC-finite (see \cite{Lyu3} or \cite[Corollary 4.2]{BBLSZ}). Let $I\supset J$ be ideals, each of which is generated by a part of a regular system of parameters. Assume $p\in J$. Then $R/J$ is a regular local ring of characteristic $p$, and it is LC-finite. By Theorem 5.2, $\Ass H_I^i(J)$ is finite for any $i\geq0$. Proposition 5.4 is a generalization of this example.
\end{example}

Let $(R, \mathfrak{m})$ be a regular local ring. If $\ch R=\ch R/\mathfrak{m}$, $R$ is said to be of \textit{equal characteristic}. If $R$ is not of equal characteristic, either
\[(\ch R, \ch R/\mathfrak{m})=(0, p),\]
or
\[(\ch R, \ch R/\mathfrak{m})=(p^n, p)\quad\text{for some $n>1$}.\]
In this case, $R$ is said to be of \textit{unequal characteristic}. Let $(R, \mathfrak{m})$ be a regular local ring of unequal characteristic, and $\ch R/\mathfrak{m}=p$. $R$ is called \textit{unramified} if $p\notin\mathfrak{m}^2$. A regular local ring of equal characteristic is conventionally referred to as unramified.

\begin{proposition}
Let $(R, \mathfrak{m})$ be an unramified regular local ring of unequal characteristic. Suppose that each of $I$ and $J$ is generated by a part of a regular system of parameters, and $p:=\ch R/\mathfrak{m}\in J$. Then, $\Ass H_I^i(J)$ is finite for any $i\geq0$.
\end{proposition}

\begin{proof}
Since $R$ is LC-finite \cite[Theorem 1]{Lyu3} and $R/J$ is a regular local ring of characteristic $p$, the claim follows by Theorem 5.2. 
\end{proof}

Next, we apply the theorem to deduce the finiteness of $\Ass H_I^i(J)$, based on the finiteness of $\Ass H_I^{i-1}(R/J)$, which is guaranteed by Marley's result\cite{Mar}:

\begin{proposition}
Let $(R, \mathfrak{m})$ be an LC-finite regular local ring of $\dim R\leq6$, and $I\supset J$ ideals of $R$. Suppose that each of $I$ and $J$ is generated by a part of a regular system of parameters. Then, $\Ass H_I^i(J)$ is finite for any $i\geq0$.
\end{proposition}

\begin{proof}
$H_I^0(J)$ is finite over $R$, and it has only a finite number of associated primes. If $\depth_J(R)=1$, for any $i\geq0$, $\Ass H_I^i(J)$ is finite by the LC-finiteness of $R$, which is isomorphic to $J$ as an $R$-module. We may assume $\depth_J(R)\geq2$, and $i>0$. Since $R$ is Cohen-Macaulay, $\hight J=\depth_J(R)\geq2$, and $\dim R/J=\dim R-\hight J\leq4$. If $\dim R/J\leq3$, $\Ass H_I^{i-1}(R/J)$ is finite for all $i>0$ by \cite[Corollary 2.6]{Mar}. If $\dim R/J=4$, since $\hight I\geq\depth_J(R)\geq2$, $\Ass H_I^{i-1}(R/J)$ is finite for all $i>0$ by \cite[Proposition 2.7]{Mar}. Hence, $\Ass H_I^i(J)$ is finite for any $i>0$, by Theorem 5.2.
\end{proof}

\begin{proposition}
Let $(R, \mathfrak{m})$ be an LC-finite regular local ring of $\dim R\leq7$, and $\ch R/\mathfrak{m}=p$. Let $I\supset J$ be ideals of $R$. Suppose that each of $I$ and $J$ is generated by a part of a regular system of parameters. Then, $\Ass H_I^i(J)$ is finite for any $i\geq0$, if either $p\in J$ or $p\notin \mathfrak{m}^2+J$.
\end{proposition}

\begin{proof}
We may assume $i>0$, $\hight J\geq2$, and $\dim R/J=\dim R-\hight J\leq5$. If $\dim R/J\leq4$, we can prove the claim using the same method as in Proposition 5.3. If $p\in J$ holds, $R/J$ is of equal characteristic; If $p\notin\mathfrak{m}^2+J$ holds, $R/J$ is unramified; in conclusion, $R/J$ is unramified. If $\dim R/J=5$, $\Ass H_I^{i-1}(R/J)$ is finite for all $i>0$, by \cite[Theorem 2.10]{Mar}. The finiteness of $\Ass H_I^i(J)$ follows by Theorem 5.2.
\end{proof}

\begin{remark}
We should emphasize in these propositions that it is not required for a regular local ring $R$ to be of equal characteristic. Indeed, if $R$ is of equal characteristic, the finiteness of $\Ass H_I^i(J)$ follows by the Lyubeznik's results (see \cite{Lyu1} and \cite{Lyu2}). 

If $R$ is of equal characteristic $p>0$ and $J$ is generated by a part of a regular system of parameters, $\Ass H_I^i(J)$ is finite for any ideal $I$ of $R$ and any $i\geq0$ \cite[Theorem 5.2]{Lew}. 

If $R$ is of equal characteristic $0$, we show that $\Ass H_I^i(J)$ is finite for any pair of ideals $I$, $J$, and for any $i\geq0$. Let $\hat{R}$ be the $\mathfrak{m}$-adic completion. $\hat{R}$ is a Noetherian local ring, and flat over $R$. Then, $H_I^i(J\hat{R})\simeq H_I^i(J)\otimes_R\hat{R}$. Furthermore, by \cite[Theorem 23.2(ii)]{Mat},
\[
\Ass_{\hat{R}} H_I^i(J\hat{R})=\bigcup_{\mathfrak{p}\in\Ass_R H_I^i(J)}\Ass_{\hat{R}} (\hat{R}/\mathfrak{p}\hat{R}).
\]
For any $\mathfrak{p}\in\Ass_R H_I^i(J)$, $\Ass_{\hat{R}}(\hat{R}/\mathfrak{p}\hat{R})$ is not empty \cite[Theorem 6.1(i)]{Mat}. The finiteness of $\Ass_{\hat{R}} H_I^i(J\hat{R})$ implies the finiteness of $\Ass_R H_I^i(J)$. We show that $\Ass_{\hat{R}} H_I^i(J\hat{R})$ is finite, using the theory of ``holonomic $D$-modules'' established by Lyubeznik\cite{Lyu1}. By Cohen's structure theorem, $\hat{R}$ is a quotient of a formal power series ring $A$ in finitely many variables, over a field of characteristic $0$. $A$ is a holonomic $D$-module \cite[(2.2a)]{Lyu1}. In addition, both submodules and quotient modules of a holonomic module inherit the property of being holonomic \cite[(2.2c)]{Lyu1}, $\hat{R}$ and $J\hat{R}$ are holonomic. By \cite[(2.2d) and Theorem 2.4(c)]{Lyu1}, $\Ass_{\hat{R}} H_I^i(J\hat{R})$ is finite.
\end{remark}

\section*{Acknowledgments}
The author is grateful to Professor Yasunari Nagai for valuable discussions and clear guidance throughout this research. The author also would like to thank the anonymous referees and the editor for their careful reading and helpful comments.


\begin{bibdiv}
\begin{biblist}

\bib{BBLSZ}{article}{
   author={Bhatt, Bhargav},
   author={Blickle, Manuel},
   author={Lyubeznik, Gennady},
   author={Singh, Anurag K.},
   author={Zhang, Wenliang},
   title={Local cohomology modules of a smooth $\Bbb{Z}$-algebra have
   finitely many associated primes},
   journal={Invent. Math.},
   volume={197},
   date={2014},
   number={3},
   pages={509--519},
   issn={0020-9910},
   review={\MR{3251828}},
   doi={10.1007/s00222-013-0490-z},
}
\bib{BrFa}{article}{
   author={Brodmann, M. P.},
   author={Faghani, A. Lashgari},
   title={A finiteness result for associated primes of local cohomology
   modules},
   journal={Proc. Amer. Math. Soc.},
   volume={128},
   date={2000},
   number={10},
   pages={2851--2853},
   issn={0002-9939},
   review={\MR{1664309}},
   doi={10.1090/S0002-9939-00-05328-4},
}
\bib{DM}{article}{
   author={Dosea, Andr\'e},
   author={Miranda-Neto, Cleto B.},
   title={On Huneke's conjecture about associated primes of local cohomology
   modules},
   journal={J. Algebra},
   volume={669},
   date={2025},
   pages={143--158},
   issn={0021-8693},
   review={\MR{4862969}},
   doi={10.1016/j.jalgebra.2025.02.008},
}
\bib{Eis}{book}{
   author={Eisenbud, David},
   title={The geometry of syzygies},
   series={Graduate Texts in Mathematics},
   volume={229},
   note={A second course in commutative algebra and algebraic geometry},
   publisher={Springer-Verlag, New York},
   date={2005},
   pages={xvi+243},
   isbn={0-387-22215-4},
   review={\MR{2103875}},
}
\bib{Hun90}{article}{
   author={Huneke, Craig},
   title={Problems on local cohomology},
   conference={
      title={Free resolutions in commutative algebra and algebraic geometry},
      address={Sundance, UT},
      date={1990},
   },
   book={
      series={Res. Notes Math.},
      volume={2},
      publisher={Jones and Bartlett, Boston, MA},
   },
   isbn={0-86720-285-8},
   date={1992},
   pages={93--108},
   review={\MR{1165320}},
}
\bib{Hun}{article}{
   author={Huneke, Craig},
   title={Lectures on local cohomology},
   note={Appendix 1 by Amelia Taylor},
   conference={
      title={Interactions between homotopy theory and algebra},
   },
   book={
      series={Contemp. Math.},
      volume={436},
      publisher={Amer. Math. Soc., Providence, RI},
   },
   isbn={978-0-8218-3814-3},
   date={2007},
   pages={51--99},
   review={\MR{2355770}},
   doi={10.1090/conm/436/08404},
}
\bib{HuSh}{article}{
   author={Huneke, Craig L.},
   author={Sharp, Rodney Y.},
   title={Bass numbers of local cohomology modules},
   journal={Trans. Amer. Math. Soc.},
   volume={339},
   date={1993},
   number={2},
   pages={765--779},
   issn={0002-9947},
   review={\MR{1124167}},
   doi={10.2307/2154297},
}
\bib{Kat}{article}{
   author={Katzman, Mordechai},
   title={An example of an infinite set of associated primes of a local
   cohomology module},
   journal={J. Algebra},
   volume={252},
   date={2002},
   number={1},
   pages={161--166},
   issn={0021-8693},
   review={\MR{1922391}},
   doi={10.1016/S0021-8693(02)00032-7},
}
\bib{Lew}{article}{
   author={Lewis, Monica A.},
   title={The local cohomology of a parameter ideal with respect to an
   arbitrary ideal},
   journal={J. Algebra},
   volume={589},
   date={2022},
   pages={82--104},
   issn={0021-8693},
   review={\MR{4321612}},
   doi={10.1016/j.jalgebra.2021.09.014},
}
\bib{Lyu1}{article}{
   author={Lyubeznik, Gennady},
   title={Finiteness properties of local cohomology modules (an application
   of $D$-modules to commutative algebra)},
   journal={Invent. Math.},
   volume={113},
   date={1993},
   number={1},
   pages={41--55},
   issn={0020-9910},
   review={\MR{1223223}},
   doi={10.1007/BF01244301},
}
\bib{Lyu2}{article}{
   author={Lyubeznik, Gennady},
   title={$F$-modules: applications to local cohomology and $D$-modules in
   characteristic $p>0$},
   journal={J. Reine Angew. Math.},
   volume={491},
   date={1997},
   pages={65--130},
   issn={0075-4102},
   review={\MR{1476089}},
   doi={10.1515/crll.1997.491.65},
}
\bib{Lyu3}{article}{
   author={Lyubeznik, Gennady},
   title={Finiteness properties of local cohomology modules for regular
   local rings of mixed characteristic: the unramified case},
   note={Special issue in honor of Robin Hartshorne},
   journal={Comm. Algebra},
   volume={28},
   date={2000},
   number={12},
   pages={5867--5882},
   issn={0092-7872},
   review={\MR{1808608}},
   doi={10.1080/00927870008827193},
}
\bib{Mar}{article}{
   author={Marley, Thomas},
   title={The associated primes of local cohomology modules over rings of
   small dimension},
   journal={Manuscripta Math.},
   volume={104},
   date={2001},
   number={4},
   pages={519--525},
   issn={0025-2611},
   review={\MR{1836111}},
   doi={10.1007/s002290170024},
}
\bib{Mat}{book}{
   author={Matsumura, Hideyuki},
   title={Commutative ring theory},
   series={Cambridge Studies in Advanced Mathematics},
   volume={8},
   edition={2},
   note={Translated from the Japanese by M. Reid},
   publisher={Cambridge University Press, Cambridge},
   date={1989},
   pages={xiv+320},
   isbn={0-521-36764-6},
   review={\MR{1011461}},
}
\bib{SiSw}{article}{
   author={Singh, Anurag K.},
   author={Swanson, Irena},
   title={Associated primes of local cohomology modules and of Frobenius
   powers},
   journal={Int. Math. Res. Not.},
   date={2004},
   number={33},
   pages={1703--1733},
   issn={1073-7928},
   review={\MR{2058025}},
   doi={10.1155/S1073792804133424},
}
\end{biblist}
\end{bibdiv}
\end{document}